\documentclass[11pt]{article}

\usepackage{amsmath,amsfonts,amssymb,amsthm,bm}
\usepackage{mathtools}
\usepackage{geometry}
\usepackage{hyperref}
\usepackage[normalem]{ulem} % \sout{...}
\usepackage{setspace}
\newtheorem{theorem}{Theorem}
\newtheorem{proposition}{Proposition}
\newtheorem{corollary}{Corollary}
\newtheorem{definition}{Definition}
\newtheorem{remark}{Remark}

\newcommand{\R}{\mathbb{R}}
\newcommand{\E}{\mathbb{E}}
\newcommand{\Cov}{\mathrm{Cov}}
\newcommand{\Range}{\mathrm{Range}}
\newcommand{\Null}{\mathrm{Null}}
\newcommand{\argmin}{\operatorname*{arg\,min}}
\newcommand{\norm}[2][]{\left\lVert #2 \right\rVert_{#1}}
\newcommand{\ip}[2]{\left\langle #1, #2 \right\rangle}

\title{Ensemble-Conditional Gaussian Processes (Ens-CGP):\\
Representation, Geometry, and Inference}
\author{Sai Ravela \and Jae Deok Kim \and Kenneth Gee \and Xingjian Yan  \and Samson Mercier \and Lubna Albarghouty  \and Anamitra Saha  }
\date{Earth Signals and Systems Group\\
Earth, Atmospheric and Planetary Sciences\\ Massachusetts Institute of Technology}

\begin{document}
\maketitle

\begin{abstract}

We formulate Ensemble-Conditional Gaussian Processes (Ens-CGP), a finite-dimensional synthesis that centers ensemble-based inference on the conditional Gaussian law. Conditional Gaussian processes (CGP) arise directly from Gaussian processes under conditioning and, in linear–Gaussian settings, define the full posterior distribution for a Gaussian prior and linear observations. Classical Kalman filtering is a recursive algorithm that computes this same conditional law under dynamical assumptions; the conditional Gaussian law itself is therefore the underlying representational object, while the filter is one computational realization. In this sense, CGP provides the probabilistic foundation for Kalman-type methods as well as equivalent formulations as a strictly convex quadratic program (MAP estimation), RKHS-regularized regression, and classical regularization.

Ens-CGP is the ensemble instantiation of this object, obtained by treating empirical ensemble moments as a (possibly low-rank) Gaussian prior and performing exact conditioning.  By separating representation (GP → CGP → Ens-CGP) from computation (Kalman filters, EnKF variants, and iterative ensemble schemes), the framework links an earlier-established representational foundation for inference to ensemble-derived priors and clarifies the relationships among probabilistic, variational, and ensemble perspectives.

\end{abstract}

\section{Introduction}\label{sec:intro}
We introduce Ensemble-Conditional Gaussian Processes (Ens-CGP), a finite-dimensional framework for Gaussian inference in which ensemble statistics define a prior law and conditioning yields a posterior distribution that unifies Gaussian processes, RKHS regression, quadratic programs, and Kalman-type updates.

A single mathematical object underlies several literatures that are often taught and practiced as if they were different:
\emph{Gaussian conditioning} (Gaussian processes and kriging) \cite{rasmussen_Regression_2005}, \emph{kernel/RKHS regularization} (splines and learning theory) \cite{wahba_MoreSplines_1990},
\emph{quadratic optimization} (MAP estimators and Tikhonov-type programs)\cite{tarantola_GeneralDiscreteInverse_2005}, and \emph{Kalman-style updates} (KF/EnKF in data assimilation)\cite{evensen_EnsembleKalmanFilter_2003, kalman_NewApproachLinear_1960}.
In the linear-Gaussian setting, these are not merely related; they are \emph{the same inference problem} written in different ways.
Yet, in large-scale practice (especially when ensembles are used), this equivalence is often left implicit, and algorithmic narratives (e.g., EnKF variants, iterative ensemble schemes) can obscure what is represented and what is approximated.

\paragraph{Core claim (equivalence chart).}
Fix a finite discretization and a linear observation model with Gaussian noise. A Gaussian prior on the state vector induces a joint Gaussian law on the state and observations, and conditioning yields a posterior Gaussian distribution obtained via a Schur-complement formula~\cite{anderson_DiscreteTimeKalmanFilter_1979, tarantola_GeneralDiscreteInverse_2005}.
This posterior law, which we call a \emph{conditional Gaussian process} (CGP) in discretized form, 
coincides exactly with the Kalman analysis posterior under linear–Gaussian assumptions: both the posterior mean \emph{and} posterior covariance are identical.

The posterior mean is simultaneously:
(i) the \emph{Kalman filter / smoother} mean update,
(ii) the \emph{unique minimizer} of a strictly convex quadratic program (the negative log-posterior), and
(iii) the solution of an \emph{RKHS-regularized regression} problem whose norm is induced by the same covariance/kernel matrix.

The posterior covariance equals the inverse (or, when singular, the Moore--Penrose pseudoinverse) of the Hessian/precision restricted to the admissible subspace, thereby linking probabilistic conditioning to optimization geometry \cite{bishop_LinearModelsClassification_2006}.
Moreover, any square-root factorization of the covariance—spectral modes, feature maps, or ensemble anomaly matrices—induces the same subspace geometry and therefore the same posterior law \cite{tippett_2003}.

This paper makes this equivalence explicit in finite-dimensional vector spaces.

\paragraph{Why the GP/CGP framing matters.}
The GP/CGP point of view is representation-first: it specifies a joint Gaussian law and therefore fixes the entire posterior distribution under conditioning.
This determines not only an estimator, but also its uncertainty and the structure within which updates occur.
In particular, conditioning a joint Gaussian law yields both a posterior mean and a posterior covariance \cite{rasmussen_Regression_2005}; the Kalman filter is a recursive implementation of this same conditioning identity in the linear–Gaussian setting \cite{sarkka_BayesianInference_2013}.
Thus, CGP is not an alternative to Kalman filtering, but the probabilistic object that Kalman filtering computes.

Unlike the classical filter formulation, however, CGP is not tied to a dynamical model or to sequential data assimilation.
It is a single conditioning operation for any chosen Gaussian prior and linear information operator (equivalently, for any jointly Gaussian model, including Gaussian graphical models).
Once the conditional law is taken as the primitive, ensemble constructions, localization, inflation, and iterative inversion schemes can be understood as approximations to or modifications of that law \cite{cockayne_BayesianProbabilisticNumerical_2019}.
This separation between representation (the conditional Gaussian law) and algorithm (the method of computation) is essential for clarity, attribution, and interpretation.

\paragraph{When CGP is the natural primary object.}
If the primary objective is probabilistic inference, conditioning, uncertainty propagation, attribution of information to data versus prior, or coupling to downstream decision-making, then CGP/Ens-CGP is the natural top-level object.
If the primary objective is deterministic regularization theory, approximation classes, rates, representer theorems, and minimax analysis, then the RKHS viewpoint can be a more convenient starting point.
The equivalence chart shows these centers coincide in the linear-Gaussian case; the choice is one of emphasis, not mathematics.

\paragraph{Historical Perspective. }
Gaussian conditioning for spatial fields appears in geostatistics as \emph{kriging} \cite{matheron_PrinciplesGeostatistics_1963} and in probability as Gaussian random fields \cite{adler_GaussianFields_2007}.
The RKHS--Bayes equivalence appears in spline theory and regularization (the Kimeldorf--Wahba/Wahba) \cite{kimeldorf_CorrespondenceBayesianEstimation_1970} and in the Moore--Aronszajn construction of RKHS from kernels \cite{Aronszajn1950TheoryOR}.
Kalman's contribution is the recursive exploitation of the same conditioning identities for linear dynamical models \cite{kalman_NewApproachLinear_1960}.
In modern machine learning, ``Gaussian process regression'' popularized the kernel-as-prior narrative \cite{rasmussen_CovarianceFunctions_2005}.
Our goal is not to relabel these results, but to expose their shared backbone in a form that is immediately usable in large-scale, discretized computation and in ensemble implementations.

\paragraph{Contribution and structure.}
The paper is organized as a single thread that starts from classical foundations only as needed and ends in a fully discrete vector space:
\begin{itemize}
\item We state the \emph{discrete equivalence chart} (CGP $\Leftrightarrow$ QP(MAP) $\Leftrightarrow$ RKHS regression $\Leftrightarrow$ KF/EnKF mean update) and prove it with appropriate singular-covariance handling and range constraints.
\item We connect kernels/covariances to geometry via square roots and Gram matrices.
\item We define \emph{Ens-CGP}: the conditional Gaussian law obtained by conditioning a Gaussian prior whose mean and covariance are defined empirically by an ensemble.
This isolates what an ensemble \emph{represents} (a low-rank prior) from how one \emph{computes} with it (EnKF variants and other schemes).
\item We explain, at the representational level, why repeatedly reusing the same observation without new information
cannot be interpreted as Kalman conditioning unless an explicit probabilistic model introducing new information is specified, clarifying where iterative ensemble inversion narratives depart from Bayesian updating.
\end{itemize}

Gaussian conditioning, Schur-complement identities, the Kalman filter, and the equivalence between Bayesian linear inverse problems, quadratic programs, and RKHS regularization are classical results \cite{sarkka_BayesianInference_2013, wahba_MoreSplines_1990}. None of these mathematical identities are new. What is new is the explicit integration of the ensemble into the conditional Gaussian framework. We introduce the \emph{Ensemble–Conditional Gaussian Process (Ens-CGP)} as a representational object: a conditional Gaussian law whose prior mean and covariance are defined empirically by an ensemble. In this formulation, ensemble Kalman methods are not primary constructs but computational realizations of an underlying conditional Gaussian law \cite{tippett_2003}. The ensemble becomes a low-rank prior representation, and conditioning yields a well-defined posterior law.

This reframing folds EnKF-type methods into CGP, rather than treating CGP as a special case of ensemble Kalman filtering. Given its historical priority claims, CGP becomes a foundational probabilistic and inference-theoretic object, with ensembles understood as finite-dimensional instantiations of it and Kalman-related methods as special cases. The contribution of this document is therefore structural: it clarifies hierarchy, separates representation from algorithm, and makes explicit the equivalence
\[
\text{CGP} \equiv \text{MAP (QP)} \equiv \text{RKHS regularization} \equiv \text{Kalman posterior (mean and covariance)},
\]
with Ens-CGP as the ensemble-instantiated form of the same backbone.

Ens-CGP serves as the common representational backbone for a family of inverse and learning problems currently under development, including equation discovery~\cite{gee_2026}, seismic phase association~\cite{lubna_2024}, magnetotelluric inversion~\cite{kim_2025}, climate downscaling~\cite{saha_2024}, surrogate modeling~\cite{mirhoseini_2024}, and astrophysical inference.

The remainder of the paper develops the minimal classical background (Riesz, covariance operators, Mercer/KL, Moore--Aronszajn) needed to justify discretization, and then returns to the finite-dimensional setting that underlies computation and ensembles.

\section*{Notation and Assumptions}
\paragraph{Probability.} All random variables are defined on a probability space $(\Omega,\mathcal{F},\mathbb{P})$. Expectation and covariance are denoted by $\E[\cdot]$ and $\Cov(\cdot,\cdot)$.
\paragraph{Index set and sampling.}
Let $\mathcal{X}$ be an index set (e.g., spatial locations), and let
$f=\{f(x):x\in\mathcal X\}$ denote a stochastic process (a representational device; computation will occur only on a fixed finite discretization $X$).
For a finite subset $X=\{x_1,\dots,x_n\}\subset \mathcal{X}$ we write
\[
\bm{f}_X := (f(x_1),\dots,f(x_n))^\top \in \R^n,\qquad
\bm{m}_X := \E[\bm{f}_X],\qquad
(\bm{K}_{XX})_{ij} := \Cov(f(x_i),f(x_j)).
\]
When the subset $X$ is fixed from context, we abbreviate
$\bm{f}:=\bm{f}_X$, $\bm{m}:=\bm{m}_X$, and $\bm{K}:=\bm{K}_{XX}$.

The formal definition of a Gaussian process as a consistent family of
finite-dimensional Gaussian vectors is given in Section~\ref{sec:gp}. Throughout, ``CGP'' denotes conditioning on such finite-dimensional jointly Gaussian vectors.

\paragraph{Linear observation model (discrete).} We consider
\begin{equation}\label{eq:linobs}
\bm{y} = \bm{H}\bm{f} + \bm{\varepsilon},\qquad
\bm{\varepsilon}\sim\mathcal{N}(\bm{0},\bm{R}),\qquad
\bm{\varepsilon}\perp \bm{f},
\end{equation}
with $\bm{H}\in\R^{m\times n}$ and $\bm{R}\in\R^{m\times m}$ symmetric positive definite (SPD) unless stated otherwise.
\paragraph{PSD and pseudoinverse.} $\bm{K}\succeq 0$ denotes symmetric positive semidefinite (PSD). $\bm{K}^\dagger$ is the Moore--Penrose pseudoinverse.
\paragraph{Weighted norms.} For SPD $\bm{A}$ we write $\norm[\bm{A}]{\bm{v}}^2 := \bm{v}^\top \bm{A}\bm{v}$.

\section{Preview: Discrete Equivalence Roadmap}\label{sec:roadmap}
Our goal is to show that several standard objects are \emph{the same} inference problem under different representations.
The thread is finite-dimensional and computational: we work with vectors, matrices, and (possibly low-rank) covariances.
The classical sections later in the paper motivate these objects, but the equivalences below are the operative end state.

\paragraph{Reading guide.}
We state the fully finite-dimensional equivalence chart first (this section), since it is the computational end state.
The classical sections that follow justify how the matrices arise from continuous objects, and we then return to purely discrete constructions for ensembles.

\begin{definition}[CGP (finite-dimensional, conditional Gaussian law)]\label{def:cgp_fd}
Let $\bm f\in\R^n$ and $\bm y\in\R^m$ be jointly Gaussian. Then $\bm f\mid \bm y$ is Gaussian, with mean and covariance given by Schur-complement conditioning formulas. In this paper, ``CGP'' names this conditional Gaussian law on a fixed discretization $X$; the process-level definition of a GP is recalled in Section~\ref{sec:gp}.
\end{definition}

\begin{theorem}[CGP = Classical Quadratic Form (MAP)]\label{thm:cgp_qp_classical}
Let $\bm{f}\in\R^n$ with prior $\bm{f}\sim\mathcal{N}(\bm{m},\bm{K})$ where $\bm{K}\succeq 0$,
and observations $\bm{y}=\bm{H}\bm{f}+\bm{\varepsilon}$ with $\bm{\varepsilon}\sim\mathcal{N}(\bm{0},\bm{R})$,
$\bm{R}\succ 0$, independent of $\bm{f}$. Define the displacement variable
\[
\bm{x}:=\bm{f}-\bm{m}\in\R^n,\qquad \bm{d}:=\bm{y}-\bm{H}\bm{m}\in\R^m.
\]
Then the negative log-posterior (up to an additive constant) is the convex quadratic
\begin{equation}\label{eq:quad_form}
\mathcal{J}(\bm{x})
=
\frac12\,\bm{x}^\top \bm{Q}\,\bm{x}
+
\bm{q}^\top \bm{x}
+
c,
\qquad \bm{x}\in\Range(\bm{K}),
\end{equation}
with
\begin{align}
\bm{Q} &:= \bm{K}^\dagger + \bm{H}^\top \bm{R}^{-1}\bm{H}\in\R^{n\times n},\label{eq:Q_def}\\
\bm{q} &:= -\bm{H}^\top \bm{R}^{-1}\bm{d}\in\R^{n},\label{eq:q_def}\\
c &:= \frac12\,\bm{d}^\top \bm{R}^{-1}\bm{d}\in\R.\label{eq:c_def}
\end{align}
If $\bm{K}\succ 0$ then $\bm{x}\in\R^n$ and $\bm{Q}\succ 0$; if $\bm{K}\succeq 0$ is singular,
the objective is finite only in $\Range(\bm{K})$.

The unique minimizer $\bm{x}_\star$ of \eqref{eq:quad_form} satisfies the normal equations
\begin{equation}\label{eq:normal_eq}
\bm{Q}\bm{x}_\star = -\bm{q},
\end{equation}
and the posterior mean is $\bm{m}_{\mathrm{post}}=\bm{m}+\bm{x}_\star$ \cite{tarantola_LeastSquaresCriterion_2005}.
Moreover, $\bm{m}_{\mathrm{post}}$ equals the conditional mean $\E[\bm{f}\mid \bm{y}]$.
\end{theorem}

\begin{theorem}[Quadratic Program = RKHS-Regularized Regression (Wahba, discrete)]\label{thm:qp_rkhs}
Let $\bm{K}\succeq 0$ and define the finite-dimensional Hilbert space
\[
\mathcal{H}_{\bm{K}} := \Range(\bm{K})\subseteq \R^n,\qquad
\ip{\bm{u}}{\bm{v}}_{\mathcal{H}_{\bm{K}}}:=\bm{u}^\top \bm{K}^\dagger \bm{v}.
\]
Then the minimizer in Theorem~\ref{thm:cgp_qp_classical} is equivalently
\[
\bm{m}_{\mathrm{post}}
=
\argmin_{\bm{g}\in \bm{m}+\mathcal{H}_{\bm{K}}}\;
\norm[\bm{R}^{-1}]{\bm{y}-\bm{H}\bm{g}}^2
+
\norm[\mathcal{H}_{\bm{K}}]{\bm{g}-\bm{m}}^2,
\]
i.e., CGP inference equals RKHS-regularized regression with the RKHS norm induced by $\bm{K}$ \cite{wahba_MoreSplines_1990}.
\end{theorem}

\begin{theorem}[Square roots, Gram matrices, and RKHS geometry (discrete)]\label{thm:rkhs_kernel}
Let $\bm{K}\in\R^{n\times n}$ be PSD with rank $r$. Then there exists a feature/square-root factor
$\bm{A}\in\R^{n\times r}$ such that
\[
\bm{K}=\bm{A}\bm{A}^\top.
\]
Define the feature matrix $\bm{\phi}:=\bm{A}^\top\in\R^{r\times n}$ and let $\bm{\phi}_i$ denote its $i$th column.
Then
\[
\bm{K}=\bm{\phi}^\top\bm{\phi},
\qquad
K_{ij}=\bm{\phi}_i^\top\bm{\phi}_j,
\]
so $\bm{K}$ is the \emph{Gram matrix} of the feature vectors $\{\bm{\phi}_i\}_{i=1}^n$ \cite{scholkopf_Kernels_2001, scholkopf_Regularization_2001}.
The dual (feature-space) Gram matrix is $\bm{G}:=\bm{\phi}\bm{\phi}^\top=\bm{A}^\top\bm{A}$.
Moreover, the induced finite-dimensional RKHS is \cite{berlinet_RKHSStochasticProcesses_2004}
\[
\mathcal{H}_{\bm{K}}=\Range(\bm{K})=\Range(\bm{A}),
\qquad
\ip{\bm{u}}{\bm{v}}_{\mathcal{H}_{\bm{K}}}=\bm{u}^\top \bm{K}^\dagger \bm{v}.
\]
\end{theorem}

\begin{corollary}[Ens-CGP and EnKF mean coincide under linear--Gaussian assumptions]\label{cor:enkf}
If $(\bar{\bm{f}},\bm{K})$ are estimated by an ensemble and the update uses $\bm{K}$ in the Kalman gain
\[
\bm{G} := \bm{K}\bm{H}^\top(\bm{H}\bm{K}\bm{H}^\top+\bm{R})^{-1},
\]
then the posterior mean update $\bar{\bm{f}} \mapsto \bar{\bm{f}} + \bm{G}(\bm{y}-\bm{H}\bar{\bm{f}})$ is exactly the conditional mean of the Ens-CGP defined in Section~\ref{sec:enscgp}. Standard EnKF updates are Monte Carlo realizations whose sample mean equals this update (up to sampling error and any additional numerical choices such as localization/inflation) \cite{evensen_EnsembleKalmanFilter_2003}.
\end{corollary}

\section{Gaussian Processes: Classical Definition}\label{sec:gp}
A Gaussian process is best understood as a family of \emph{consistent} finite-dimensional Gaussian vectors.
This is both the formal definition and the justification for our discretization-first stance.

\subsection{Definition (process as consistent finite-dimensional Gaussians)}
A \emph{Gaussian process (GP)} is a collection of real random variables $\{f(x):x\in\mathcal{X}\}$ such that for any finite subset $X=\{x_1,\dots,x_n\}$, the vector $\bm{f}_X$ is multivariate normal:
\[
\bm{f}_X \sim \mathcal{N}(\bm{m}_X,\bm{K}_{XX}),
\quad
(\bm{m}_X)_i=\E[f(x_i)],\quad
(\bm{K}_{XX})_{ij}=\Cov(f(x_i),f(x_j)).
\]

\subsection{Consistency and Kolmogorov extension}\label{sec:kolmogorov}
A family $\{\mathcal{N}(\bm{m}_X,\bm{K}_{XX})\}_X$ defines a stochastic process if it is consistent under marginalization:
for $Y\subset X$, the marginal of $\mathcal{N}(\bm{m}_X,\bm{K}_{XX})$ onto the coordinates indexed by $Y$
equals $\mathcal{N}(\bm{m}_Y,\bm{K}_{YY})$.
For Gaussian families, this holds automatically because marginals of multivariate Gaussians are Gaussian with the corresponding subvector and submatrix \cite{rasmussen_Regression_2005}.

\begin{theorem}[Kolmogorov extension theorem]
Let $\{P_X\}_{X\subset\mathcal X,\ |X|<\infty}$ be a family of probability measures on $\R^{|X|}$ that is consistent under marginalization.
Then there exists a stochastic process $\{f(x):x\in\mathcal X\}$ on some $(\Omega,\mathcal F,\mathbb P)$ such that for every finite $X=\{x_1,\dots,x_n\}$ \cite{billingsley_StochasticProcesses_2012},
\[
(f(x_1),\dots,f(x_n)) \sim P_X.
\]
\end{theorem}

\begin{remark}[Why we emphasize the discrete setting]
Kolmogorov extension guarantees the existence of a process, but it does not guarantee sample-path regularity (continuity, differentiability, etc.).
Since all computation in this paper occurs on a fixed finite discretization $X$, we treat GPs as families of finite-dimensional Gaussian vectors and regard
the infinite-dimensional process as a representational device rather than an object of computation. Throughout, we ultimately restrict to a fixed finite set $X$, so the GP is fully represented by the random vector $\bm{f}\in\R^n$ and parameters $(\bm{m},\bm{K})$.
\end{remark}

% \subsection{Discrete emphasis}
% Throughout we ultimately restrict to a fixed finite set $X$, so the GP is fully represented by the random vector $\bm{f}\in\R^n$ and parameters $(\bm{m},\bm{K})$.

\section{Linear Functionals and Riesz Representation}\label{sec:riesz}
This section explains how continuous linear measurements give rise to inner products and why this matters for covariance.
The payoff is the operator identity that discretizes into matrix covariances \cite{steinwart_KernelsReproducingKernel_2008}.

\subsection{Hilbert space and linear functionals}
Let $\mathcal{H}$ be a real Hilbert space and $\ell:\mathcal{H}\to\R$ a bounded linear functional.

\begin{theorem}[Riesz representation]\label{thm:riesz}
There exists a unique $h_\ell\in\mathcal{H}$ such that $\ell(g)=\ip{g}{h_\ell}_{\mathcal{H}}$ for all $g\in\mathcal{H}$ \cite{reed_IIHilbertSpaces_1972}.
\end{theorem}

\subsection{Covariance operator}
Let $f$ be an $\mathcal{H}$-valued random element with mean $m=\E[f]$ (Bochner integral). Define the rank-one (outer-product) operator
\[
(u\otimes v)\,g := u\,\ip{v}{g}_{\mathcal H},\qquad u,v,g\in\mathcal H.
\]
The covariance operator $\mathcal{C}:\mathcal{H}\to\mathcal{H}$ is
\[
\mathcal{C} := \E\big[(f-m)\otimes(f-m)\big],
\]
equivalently characterized by
\[
\ip{h}{\mathcal{C}g}_{\mathcal{H}}
=
\E\Big[\ip{h}{f-m}_{\mathcal{H}}\;\ip{f-m}{g}_{\mathcal{H}}\Big],
\qquad \forall\,h,g\in\mathcal H.
\]
Then, for bounded linear functionals $\ell_1,\ell_2$ with representers $h_{\ell_1},h_{\ell_2}$,
\[
\Cov(\ell_1(f),\ell_2(f))=\ip{h_{\ell_1}}{\mathcal{C}h_{\ell_2}}_{\mathcal{H}}.
\]
\begin{remark}
This operator identity is the continuous analogue of matrix covariance \cite{adler_3SampleFunction_2010}. It is the correct object to diagonalize (Mercer/KL) and to discretize (Section~\ref{sec:disc}).
\end{remark}

\section{Mercer Theory and Karhunen--Lo\`eve Expansion}\label{sec:mercer}
We connect kernels/covariances to spectral decompositions. In the discrete setting, this becomes eigen/SVD factorization, which also underlies ensemble square roots \cite{tippett_2003}.

\subsection{Integral operator}
Assume $\mathcal{X}$ is compact with measure $\mu$ and let $k$ be continuous, symmetric, positive definite. Define $T_k:L^2(\mathcal{X},\mu)\to L^2(\mathcal{X},\mu)$ by
\[
(T_k g)(x) = \int_{\mathcal{X}} k(x,x')\, g(x')\, d\mu(x').
\]
Then $T_k$ is compact, self-adjoint, and positive.

\subsection{Mercer expansion}
Assume $k$ is continuous, symmetric, and positive definite on a compact
space $\mathcal X$ equipped with a finite Borel measure $\mu$.
Then Mercer’s theorem yields eigenpairs $(\lambda_i,\psi_i)$ with
$\lambda_i\ge 0$ and $\{\psi_i\}$ orthonormal in $L^2(\mathcal X,\mu)$
such that
\[
k(x,x')
=
\sum_{i=1}^{\infty}
\lambda_i\,\psi_i(x)\psi_i(x'),
\]
where the series converges absolutely and uniformly on
$\mathcal X\times\mathcal X$
\cite{adler_3SampleFunction_2010, berlinet_RKHSStochasticProcesses_2004}.

\subsection{Karhunen--Lo\`eve (KL) representation}
If $f$ is a centered GP on $\mathcal X$ with covariance $k$ and the mild integrability condition
\[
\int_{\mathcal X} k(x,x)\,d\mu(x)<\infty
\]
holds, then
\[
f(x)=\sum_{i=1}^\infty \sqrt{\lambda_i}\, z_i\,\psi_i(x),
\qquad z_i\stackrel{iid}{\sim}\mathcal{N}(0,1),
\]
where $(\lambda_i,\psi_i)$ are the eigenpairs of the covariance (integral) operator $T_k$ on
$L^2(\mathcal{X},\mu)$. The series converges in $L^2(\Omega)$ for each fixed $x\in\mathcal X$
(and under stronger conditions, in pointwise or uniform senses)
\cite{adler_3SampleFunction_2010, ghanem_RepresentationStochasticProcesses_1991}.

\section{Moore--Aronszajn and RKHS}\label{sec:rkhs}
This section explains why kernels define Hilbert-space geometry and reproducing properties \cite{Aronszajn1950TheoryOR}.
In the discrete setting, RKHS geometry is precisely the geometry induced by $\bm K^\dagger$ on $\Range(\bm K)$ \cite{berlinet_RKHSStochasticProcesses_2004}.

\subsection{Positive definite kernels}
A symmetric function $k:\mathcal{X}\times\mathcal{X}\to\R$ is positive definite if for all finite $\{x_i\}$ and coefficients $\{a_i\}$,
\[
\sum_{i,j} a_i a_j k(x_i,x_j)\ge 0.
\]

\subsection{Moore--Aronszajn}
\begin{theorem}[Moore--Aronszajn]\label{thm:ma}
For each positive definite kernel $k$ there exists a unique RKHS $\mathcal{H}_k$ of functions on $\mathcal{X}$ such that $k(\cdot,x)\in\mathcal{H}_k$ and the reproducing property holds \cite{Aronszajn1950TheoryOR}:
\[
f(x) = \ip{f}{k(\cdot,x)}_{\mathcal{H}_k}\quad \forall f\in\mathcal{H}_k.
\]
\end{theorem}

\subsection{Feature map and Gram matrix}
Define the canonical feature map $\phi(x):=k(\cdot,x)\in\mathcal{H}_k$. Then
\[
k(x,x')=\ip{\phi(x)}{\phi(x')}_{\mathcal{H}_k}.
\]
For a finite set $X=\{x_1,\dots,x_n\}$ the Gram matrix satisfies
\[
(\bm{K}_{XX})_{ij}=k(x_i,x_j)=\ip{\phi(x_i)}{\phi(x_j)}_{\mathcal{H}_k}.
\]
In the fully discrete vector setting we may select a square-root $\bm{A}$ with $\bm{K}_{XX}=\bm{A}\bm{A}^\top$ and define $\bm{\phi}:=\bm{A}^\top$, so that $\bm{K}_{XX}=\bm{\phi}^\top\bm{\phi}$ (Theorem~\ref{thm:rkhs_kernel}) \cite{scholkopf_Kernels_2001, berlinet_RKHSStochasticProcesses_2004}.

\section{Wahba Equivalence in Finite Dimensions}\label{sec:wahba}
We now make the core inference equivalence explicit in the discrete setting.
This is the operational bridge between probability (conditioning), optimization (MAP), and geometry (RKHS) \cite{ohagan_CurveFittingOptimal_1978, wahba_MoreSplines_1990}.

\subsection{Prior and likelihood (linear--Gaussian)}
Let $\bm{f}\in\R^n$ and assume the Gaussian prior
\[
\bm{f}\sim\mathcal{N}(\bm{m},\bm{K}),\qquad \bm{K}\succeq 0,
\]
and the observation model \eqref{eq:linobs}. Then the joint vector
\[
\begin{bmatrix}\bm{f}\\\bm{y}\end{bmatrix}
\]
is Gaussian with
\[
\E[\bm{y}] = \bm{H}\bm{m},\qquad
\Cov(\bm{y},\bm{y})=\bm{H}\bm{K}\bm{H}^\top+\bm{R},\qquad
\Cov(\bm{f},\bm{y})=\bm{K}\bm{H}^\top.
\]

\subsection{CGP: Conditional Gaussian law in vectors}
\noindent This section instantiates Definition~\ref{def:cgp_fd} for the linear--Gaussian model \eqref{eq:linobs}.
Standard Gaussian conditioning yields \cite{rasmussen_Regression_2005, sarkka_BayesianInference_2013}
\begin{equation}\label{eq:postmean}
\bm{m}_{\mathrm{post}}
=
\bm{m}+\bm{K}\bm{H}^\top(\bm{H}\bm{K}\bm{H}^\top+\bm{R})^{-1}(\bm{y}-\bm{H}\bm{m})
\end{equation}
and
\begin{equation}\label{eq:postcov}
\bm{K}_{\mathrm{post}}
=
\bm{K}-\bm{K}\bm{H}^\top(\bm{H}\bm{K}\bm{H}^\top+\bm{R})^{-1}\bm{H}\bm{K}.
\end{equation}

\paragraph{Posterior covariance as inverse Hessian (discrete CGP).}
In the linear--Gaussian setting with prior $\bm f\sim\mathcal N(\bm m,\bm K)$
(possibly singular) and observation model
$\bm y=\bm H\bm f+\bm\varepsilon$, $\bm\varepsilon\sim\mathcal N(\bm 0,\bm R)$,
the negative log-posterior (up to an additive constant) is the quadratic form
\[
\mathcal J(\bm f)
=
\frac12\|\bm f-\bm m\|_{\bm K^\dagger}^2
+
\frac12\|\bm y-\bm H\bm f\|_{\bm R^{-1}}^2,
\qquad \bm f\in \bm m+\Range(\bm K).
\]
Its Hessian (precision on the admissible subspace) is
\[
\bm Q
=
\bm K^\dagger
+
\bm H^\top \bm R^{-1}\bm H.
\]
The posterior covariance equals the inverse of $\bm Q$ on $\Range(\bm K)$; equivalently,
$\bm K_{\mathrm{post}}$ is the Moore--Penrose pseudoinverse of $\bm Q$ restricted to $\Range(\bm K)$.
This agrees with the Schur-complement covariance obtained by conditioning a joint Gaussian law.

\subsection{Proof of CGP = QP(MAP) = RKHS (finite-dimensional)}
\begin{proof}
Write $\bm{f}=\bm{m}+\bm{x}$ and $\bm{d}=\bm{y}-\bm{H}\bm{m}$ so that
$\bm{y}-\bm{H}\bm{f}=\bm{d}-\bm{H}\bm{x}$.
Up to an additive constant independent of $\bm{x}$, the negative log-posterior equals
\[
\mathcal{J}(\bm{x})
=
\frac12 \norm[\bm{R}^{-1}]{\bm{d}-\bm{H}\bm{x}}^2
+
\frac12 \norm[\bm{K}^\dagger]{\bm{x}}^2,
\qquad \bm{x}\in\Range(\bm{K}),
\]
where the domain restriction is necessary when $\bm{K}$ is singular.

Expanding gives
\[
\mathcal{J}(\bm{x})
=
\frac12 \bm{x}^\top(\bm{K}^\dagger+\bm{H}^\top\bm{R}^{-1}\bm{H})\bm{x}
+
(-\bm{H}^\top\bm{R}^{-1}\bm{d})^\top \bm{x}
+
\frac12 \bm{d}^\top\bm{R}^{-1}\bm{d},
\]
which is exactly \eqref{eq:quad_form} with \eqref{eq:Q_def}--\eqref{eq:c_def}.
The first-order optimality condition on $\Range(\bm{K})$ is $\bm{Q}\bm{x}_\star=-\bm{q}$.

Because the model is linear--Gaussian, the posterior is Gaussian and its mean equals the MAP
(minimizer of the negative log-posterior), hence $\bm{m}_{\mathrm{post}}=\bm{m}+\bm{x}_\star=\E[\bm{f}\mid\bm{y}]$.
Finally, the RKHS-regularized form in Theorem~\ref{thm:qp_rkhs} follows from the identity
$\norm[\mathcal{H}_{\bm{K}}]{\bm{g}-\bm{m}}^2=\norm[\bm{K}^\dagger]{\bm{g}-\bm{m}}^2$ on $\bm{m}+\Range(\bm{K})$ (and $+\infty$ outside) \cite{berlinet_RKHSStochasticProcesses_2004}.
\end{proof}

\section{Discretization to Vector Spaces}\label{sec:disc}
This section makes explicit how continuous theory becomes matrices, and why the operator-level identities above reduce to finite-dimensional covariances.

\paragraph{Point evaluation discretization. }
Given $X=\{x_1,\dots,x_n\}$ define $\bm{f}=(f(x_1),\dots,f(x_n))^\top$ and
$\bm{K}_{ij}=k(x_i,x_j)$. Then $\bm{f}\sim\mathcal{N}(\bm{m},\bm{K})$ exactly \cite{rasmussen_Regression_2005}.

\paragraph{General linear functional discretization. }
Let $\ell_i$ be bounded linear functionals (e.g., averages over sensor footprints). Define
\[
\bm{f}_\ell := (\ell_1(f),\dots,\ell_n(f))^\top.
\]
If $f$ is a GP (or Gaussian random element) and $\ell_i$ are bounded, then $\bm{f}_\ell$ is Gaussian with
\[
(\bm{K}_\ell)_{ij}=\Cov(\ell_i(f),\ell_j(f))=\ip{h_{\ell_i}}{\mathcal{C}h_{\ell_j}}_{\mathcal{H}},
\]
using Riesz representers and the covariance operator from Section~\ref{sec:riesz}.
This is the matrix covariance on discretized measurements \cite{tarantola_LeastSquaresCriterion_2005}.

\section{Square Roots, Modes, and Equivalent Representations}\label{sec:squareroot}
All kernel/covariance information may be represented via square roots \cite{tippett_2003}.
This is the bridge between KL/Mercer modes, SVD/EOF modes, and ensemble anomaly matrices \cite{jolliffe_PrincipalComponentAnalysis_2016}.

\subsection{Spectral factorization}
For PSD $\bm{K}$, let $\bm{K}=\bm{U}\bm{\Lambda}\bm{U}^\top$ with $\bm{\Lambda}=\mathrm{diag}(\lambda_1,\dots,\lambda_r,0,\dots)$ and $\lambda_i>0$ for $i\le r$. A canonical square root is
\[
\bm{K}^{1/2}:=\bm{U}_r\bm{\Lambda}_r^{1/2}\in\R^{n\times r},\qquad \bm{K}=\bm{K}^{1/2}(\bm{K}^{1/2})^\top.
\]

\subsection{Any square root and SVD}
Any $\bm{A}\in\R^{n\times r}$ with $\bm{K}=\bm{A}\bm{A}^\top$ is a valid square-root representation (e.g., ensemble anomalies, Cholesky factors, randomized sketches). Its thin SVD $\bm{A}=\bm{U}_r\bm{\Sigma}\bm{V}^\top$ implies $\bm{K}=\bm{U}_r\bm{\Sigma}^2\bm{U}_r^\top$. Thus ``ensemble square roots'' and ``EOF/SVD modes'' differ only by rotation in feature space \cite{livings_UnbiasedEnsembleSquare_2008}.

\begin{proof}[Proof of Theorem~\ref{thm:rkhs_kernel}]
Let $\bm{K}\succeq 0$ and take an eigendecomposition $\bm{K}=\bm{U}_r\bm{\Lambda}_r\bm{U}_r^\top$ on its rank-$r$ range. Set $\bm{A}:=\bm{U}_r\bm{\Lambda}_r^{1/2}$ so that $\bm{K}=\bm{A}\bm{A}^\top$. Defining $\bm{\phi}:=\bm{A}^\top$ yields $\bm{K}=\bm{\phi}^\top\bm{\phi}$ with entries $K_{ij}=\bm{\phi}_i^\top\bm{\phi}_j$. Finally, $\Range(\bm{K})=\Range(\bm{A})$ and the inner product $\bm{u}^\top\bm{K}^\dagger\bm{v}$ induces the RKHS geometry on that range.
\end{proof}

\section{Ensemble Approximation}\label{sec:ensemble}
We now replace expectations by empirical averages. This step turns a GP/CGP into an ensemble-instantiated prior/conditional law.

\subsection{Empirical mean and covariance}
Given an ensemble $\{\bm{f}^{(e)}\}_{e=1}^E\subset\R^n$, define
\[
\bar{\bm{f}}:=\frac{1}{E}\sum_{e=1}^E \bm{f}^{(e)},\qquad
\bm{K}:=\frac{1}{E-1}\sum_{e=1}^E(\bm{f}^{(e)}-\bar{\bm{f}})(\bm{f}^{(e)}-\bar{\bm{f}})^\top.
\]
Let the anomaly matrix be
\[
\bm{A}:=\frac{1}{\sqrt{E-1}}\big[\bm{f}^{(1)}-\bar{\bm{f}},\dots,\bm{f}^{(E)}-\bar{\bm{f}}\big]\in\R^{n\times E}.
\]
Then
\[
\bm{K}=\bm{A}\bm{A}^\top,\qquad \mathrm{rank}(\bm{K})\le E-1,\qquad \Range(\bm{K})=\Range(\bm{A}).
\]
Define the induced feature matrix
\[
\bm{\phi}:=\bm{A}^\top\in\R^{E\times n},
\]
so that $\bm{K}=\bm{\phi}^\top\bm{\phi}$ is a Gram matrix over the discretization indices \cite{hofmann_KernelMethodsMachine_2008}.

This identifies the induced discrete RKHS as $\mathcal{H}_{\bm{K}}=\Range(\bm{A})$.
\paragraph{Empirical covariance notation.}
We denote the empirical covariance by $\bm K$ and use $\widehat{\bm K}$ interchangeably
when emphasizing its role as an estimate of the true prior covariance; no distinction is intended.

\section{Ensemble--Conditional Gaussian Process (Ens-CGP)}\label{sec:enscgp}
We now define the central object in a purely discrete, tractable way: an ensemble-defined Gaussian prior and its conditional Gaussian law.

\begin{definition}[Ens-CGP]\label{def:enscgp}
Fix a finite discretization $X=\{x_1,\dots,x_n\}$. Let an ensemble
$\{\bm{f}^{(e)}\}_{e=1}^E\subset\R^n$ define an empirical mean $\bar{\bm{f}}$ and empirical covariance $\bm{K}$ as in Section~\ref{sec:ensemble}. Under the linear observation model \eqref{eq:linobs}, the \emph{Ensemble--Conditional Gaussian Process (Ens-CGP)} is the conditional Gaussian distribution
\[
\bm{f}\mid \bm{y}\sim \mathcal{N}(\bm{m}_{\mathrm{post}},\bm{K}_{\mathrm{post}})
\]
obtained by applying \eqref{eq:postmean}--\eqref{eq:postcov} with $(\bm{m},\bm{K})=(\bar{\bm{f}},\bm{K})$ and gain
\[
\bm{G}:=\bm{K}\bm{H}^\top(\bm{H}\bm{K}\bm{H}^\top+\bm{R})^{-1}.
\]
\end{definition}

\paragraph{Well-posedness of the gain.}
Even when the prior covariance $\bm K$ is low rank (as in ensemble
approximations), the matrix
\[
\bm H \bm K \bm H^\top + \bm R
\]
is symmetric positive definite because $\bm R \succ 0$.
Hence, the inverse exists, and the gain operator is well-defined without
requiring a pseudoinverse.

\subsection{Attribution and geometry}
The posterior mean shift satisfies
\[
\bm{m}_{\mathrm{post}}-\bar{\bm{f}}=\bm{K}\bm{H}^\top(\bm{H}\bm{K}\bm{H}^\top+\bm{R})^{-1}(\bm{y}-\bm{H}\bar{\bm{f}})
\in \Range(\bm{K})=\Range(\bm{A}).
\]
Thus, posterior adjustments are confined to the ensemble span; components in $\Null(\bm{K})$ are undetectable under the Ens-CGP prior.

\subsection{Connection to Kalman/EnKF (analysis step)}\label{sec:enscgp_enkf}

Under the linear-Gaussian model, conditioning yields \eqref{eq:postmean}--\eqref{eq:postcov}; these are exactly the Kalman \emph{analysis} (measurement update) formulas for the posterior mean and covariance \cite{kalman_NewApproachLinear_1960, sarkka_BayesianInference_2013}.

In a full Kalman filter, a \emph{forecast} step propagates $(\bm m,\bm K)$ through a dynamical model before the analysis step, and new observations arrive sequentially.
Ens-CGP corresponds to the analysis/conditioning step given a specified prior $(\bar{\bm f},\bm K)$; when $\bm K$ is estimated from an ensemble, the resulting mean update coincides with the standard EnKF mean update (Corollary~\ref{cor:enkf}) \cite{evensen_EnsembleKalmanFilter_2003}. Additional algorithmic devices, such as localization/inflation, modify $\bm{K}$ and thus modify the implied Ens-CGP prior; they are not part of the definition \cite{anderson_2001, hamill_2001}. 

\section{Prioritizing GP, CGP, and Ens-CGP over EnKF and EKI}\label{sec:why_gp_cgp}
This section motivates the representation-first stance in operational terms: CGP/Ens-CGP defines a complete probabilistic object (mean, covariance, geometry, and a quadratic objective),
whereas many iterative ensemble schemes should be viewed as optimization heuristics whose probabilistic interpretation is not fixed.

\subsection{Kalman as a special case of CGP}
The Kalman filter is commonly presented as a state-estimation algorithm for linear dynamical systems \cite{anderson_DiscreteTimeKalmanFilter_1979}.
CGP is the conditional law of a Gaussian random vector under linear information; it yields, simultaneously, (i) the posterior mean \emph{and covariance}, (ii) the exact MAP quadratic program, and (iii) its analytic gradients and curvature. Thus, while the Kalman filter provides covariance updates, the CGP is the representational statement whose recursion (when combined with a dynamical forecast model and sequential data) is the Kalman filter.

Define the negative log-posterior (up to an additive constant)
\begin{equation}\label{eq:nlp}
\mathcal{J}(\bm{g})
:=\tfrac12\norm[\bm{R}^{-1}]{\bm{y}-\bm{H}\bm{g}}^2+\tfrac12\norm[\bm{K}^\dagger]{\bm{g}-\bm{m}}^2,
\qquad \bm{g}\in\bm{m}+\Range(\bm{K}).
\end{equation}
Then the gradient and Hessian on $\bm{m}+\Range(\bm{K})$ are
\begin{align}
\nabla \mathcal{J}(\bm{g})
&=
-\bm{H}^\top\bm{R}^{-1}(\bm{y}-\bm{H}\bm{g})+\bm{K}^\dagger(\bm{g}-\bm{m}),
\label{eq:grad}\\
\nabla^2 \mathcal{J}(\bm{g})
&=
\bm{H}^\top\bm{R}^{-1}\bm{H}+\bm{K}^\dagger,
\label{eq:hess}
\end{align}
which is constant in $\bm{g}$ (a hallmark of the linear--Gaussian case).
Hence, CGP provides not only an update rule, but the full quadratic landscape behind it, together with the posterior covariance \eqref{eq:postcov}.

\subsection{Ens-CGP: the ensemble as a prior}
Ens-CGP replaces $(\bm{m},\bm{K})$ with ensemble statistics $(\bar{\bm{f}},\bm{K})$.
This turns the ensemble into a concrete (typically low-rank) Gaussian prior, and conditioning yields a well-defined posterior.
Attribution statements (what changed and in what subspace) become statements about the induced covariance geometry and its range.

\subsection{Ensembles in iterative inversion: surrogates, not Kalman properties}\label{sec:eki_geometry}
Iterative ensemble inversion methods (often grouped under ``EnKI/EKI'') are best interpreted as optimization schemes that reuse an empirical covariance as a low-rank metric/preconditioner \cite{iglesias_EnsembleKalmanMethods_2013}.

In this role, the ensemble primarily provides approximations to gradients and Gauss--Newton curvature of a data-misfit objective, rather than samples from a fixed Bayesian posterior \cite{schillings_AnalysisEnsembleKalman_2017}.
Absent an explicit probabilistic model that introduces genuinely new information at each step, such iterations do not inherit Kalman filtering properties: there is no single joint Gaussian law being recursively conditioned, and the evolving ensemble covariance cannot, in general, be identified with a posterior covariance under a fixed likelihood \cite{garbuno-inigo_InteractingLangevinDiffusions_2020}.

\subsection{Repeated reuse of a fixed observation is not Bayesian conditioning}\label{sec:eki_collapse}
A central interpretational issue for iterative schemes that repeatedly reuse the same observation $\bm y$
is that Bayesian conditioning is a \emph{single} update: $\bm f|\bm y$.
Unless one posits \emph{new independent information} at each iteration, repeated reuse of $\bm y$ amounts to double-counting \cite{mackay_InformationBasedObjectiveFunctions_1992}.

\begin{proposition}[Posterior collapse under repeated identical ``observations'']\label{prop:repeat_obs_collapse}
Assume a Gaussian prior $\bm f\sim \mathcal N(\bm m_0,\bm K_0)$ with $\bm K_0\succ 0$ and a linear observation model
$\bm y=\bm H\bm f+\bm\varepsilon$, $\bm\varepsilon\sim\mathcal N(\bm 0,\bm R)$ with $\bm R\succ 0$.
If one (incorrectly) treats the \emph{same} realized $\bm y$ as $k$ independent observations with identical $(\bm H,\bm R)$, then the resulting
posterior covariance and mean equal
\begin{equation}\label{eq:kfold_precision}
\bm K_k = \left(\bm K_0^{-1}+k\,\bm H^\top \bm R^{-1}\bm H\right)^{-1},\qquad
\bm m_k = \bm K_k\left(\bm K_0^{-1}\bm m_0 + k\,\bm H^\top\bm R^{-1}\bm y\right).
\end{equation}
In particular, for $\bm H=\bm I$ one has $\bm K_k=(\bm K_0^{-1}+k\,\bm R^{-1})^{-1}\to \bm 0$ and $\bm m_k\to \bm y$ as $k\to\infty$.
\end{proposition}

\begin{remark}[Implication for iterative ensemble schemes]
Equation \eqref{eq:kfold_precision} is \emph{not} a property of correct Bayesian inference for a single observation.
It is the consequence of repeatedly treating the same data realization as if it were fresh information.
Any algorithm that repeatedly reuses $\bm y$ must therefore be interpreted as a form of \emph{regularized optimization} (or as inference under an explicitly modified data model),
which immediately explains the prominence of stopping criteria, damping, and other numerical controls \cite{schillings_AnalysisEnsembleKalman_2017}.
This is the sense in which such schemes are not ``Kalman'' unless their iteration is tied to a well-defined probabilistic model that introduces new information at each step \cite{garbuno-inigo_InteractingLangevinDiffusions_2020}.
From this Bayesian perspective, the ensemble in such iterative schemes serves as a device for low-rank geometry (preconditioning/gradient approximation), rather than as a guarantee of Kalman-style posterior propagation.
\end{remark}

\section{Historical Notes: Conditional Gaussian Laws and ``CGP''}\label{sec:history_cgp}
Conditioning of Gaussian vectors and fields predates modern computational treatments; the object is classical even if the acronym ``CGP'' is not.

\paragraph{Conditioning as the primitive}
If a joint vector $\begin{bmatrix}\bm{f}^\top & \bm{y}^\top\end{bmatrix}^\top$ is Gaussian, then $\bm{f}\mid\bm{y}$ is Gaussian,
with mean and covariance given by Schur-complement formulas \eqref{eq:postmean}--\eqref{eq:postcov}.
This is standard multivariate normal theory.

\paragraph{From Gaussian random fields to kriging}
When $\bm{f}$ is interpreted as a Gaussian random field over space, conditioning on noisy observations produces the classical kriging predictor \cite{cressie_SpatialPredictionKriging_1993}, which is exactly the conditional mean of a Gaussian model with covariance specified by a kernel.

\paragraph{RKHS and spline regularization}
Independently, the equivalence between Gaussian conditioning and penalized least squares appears in smoothing spline theory and the associated RKHS formalism \cite{kimeldorf_CorrespondenceBayesianEstimation_1970, wahba_MoreSplines_1990}.
In discrete vector language, this is exactly Theorem~\ref{thm:cgp_qp_classical} and Theorem~\ref{thm:qp_rkhs}.

\paragraph{Modern ML terminology}
Gaussian process regression popularized the view of kernels as priors and posterior prediction as conditioning \cite{rasmussen_Regression_2005}.
In that literature, ``conditioning a GP'' is the standard operation; here CGP names the resulting conditional Gaussian law in the discretized vector-space setting.

\section{The Preferred Construct: CGP or RKHS?}\label{sec:glue}
The equivalence chart shows that CGP (probability) and RKHS regularization (geometry) are two faces of the same finite-dimensional object in the linear--Gaussian setting.
The distinction is therefore a matter of \emph{emphasis} and \emph{auditability}.

\subsection{Two equivalent centers with different emphases}
\paragraph{RKHS first (geometry-first).}
Given $\bm{K}\succeq 0$, the induced Hilbert space $\mathcal{H}_{\bm{K}}=\Range(\bm{K})$ with inner product
$\ip{\bm{u}}{\bm{v}}_{\mathcal{H}_{\bm{K}}}=\bm{u}^\top\bm{K}^\dagger\bm{v}$ makes the kernel/covariance into geometry.
Regularization, representer structure, and feature decompositions are most naturally described in this language.

\paragraph{CGP first (inference-first).}
CGP specifies the entire conditional distribution
\[
\bm{f}\mid\bm{y}\sim\mathcal{N}(\bm{m}_{\mathrm{post}},\bm{K}_{\mathrm{post}}).\]
This includes uncertainty quantification through $\bm{K}_{\mathrm{post}}$, and it canonically defines the quadratic MAP objective \eqref{eq:nlp}
together with analytic gradients \eqref{eq:grad} and curvature \eqref{eq:hess}.
Thus, CGP unifies estimation, uncertainty, attribution, and optimization in one object.

In ensemble-based workflows, the object one can explicitly name and audit is the implied prior law
$\mathcal{N}(\bar{\bm{f}},\widehat{\bm{K}})$ and its conditional law under $(\bm H,\bm R)$.
\paragraph{The Preferred Construct} Ens-CGP is the direct representational primitive: it uniquely determines (i) the mean update, (ii) the posterior covariance,
(iii) the induced RKHS geometry, and (iv) the MAP program.
RKHS remains essential, but as a derived structure: it describes the geometry induced by the ensemble covariance.

However, if the primary goal is deterministic regularization theory (rates, approximation classes, representer theorems) and uncertainty is secondary,
then RKHS may be the more natural top concept. Whenever uncertainty, conditioning, and attribution are primary, CGP/Ens-CGP is the more fundamental frame.

\section{Conclusion}
The paper isolates a discrete representational backbone:
\begin{align}
\text{CGP} &\equiv \text{Quadratic Program (MAP)} \nonumber\\
&\equiv \text{RKHS-regularized regression} \nonumber\\
&\equiv \text{conditioning with covariance/kernel } \bm K. \nonumber
\end{align}
Ens-CGP is the ensemble-instantiated version of the same object: a low-rank Gaussian prior and its conditional Gaussian law.
This separates representation from algorithms and clarifies what ensemble methods can and cannot represent (geometry, attribution subspaces, and rank limits).

We argue that GP/CGP provide a representation-first framework for linear-Gaussian inference, in which Kalman filters and their ensemble variants arise as computational realizations of Gaussian conditioning.
Ens-CGP is the ensemble-instantiated conditional Gaussian law obtained by treating empirical ensemble moments as a (typically low-rank) Gaussian prior and then conditioning under $(\bm H,\bm R)$.
This clarifies what is represented (a prior law and its conditional) versus what is algorithmic (EnKF variants, localization, inflation, and iterative inversion schemes), and it makes explicit when repeated reuse of fixed data departs from Bayesian conditioning.
The tight connection with RKHS and geometric square-root representations makes the framework transferable across inverse problems, data assimilation, and kernel-based regression.

\section*{Acknowledgment}
This research is part of the MIT Climate Grand Challenge, Jameel Observatory CREWSNet, and Weather and Climate Extremes projects. Schmidt Sciences, LLC, and the Bangkok Bank generously supported this project.
The authors used OpenAI ChatGPT as a writing assistant for language polishing and to suggest edits to exposition. The authors also used Grammarly for editing. All technical content, claims, and errors remain the responsibility of the authors.

\newpage
\bibliographystyle{plain}
\bibliography{enscgp_bibliography}

\end{document}